\documentclass{elsarticle} 
\usepackage{amsmath}
\usepackage{amssymb}
\usepackage{amsmath}
\usepackage{amsthm}
\usepackage{cleveref}
\usepackage{graphicx}
\usepackage{caption}
\usepackage{epstopdf}
\usepackage{wrapfig}
\usepackage{url}

\usepackage{blkarray}

\usepackage{algorithm}
\usepackage[noend]{algpseudocode}

\usepackage{color}
\usepackage{comment}
\usepackage{multirow}
\usepackage[table,xcdraw]{xcolor}

\newtheorem{theorem}{Theorem}

\newtheorem{lemma}[theorem]{Lemma}

\begin{document}

\title{A divide-and-conquer algorithm for binary matrix completion}
	
\author{Melanie Beckerleg\fnref{fn1}}
\ead{beckerleg@maths.ox.ac.uk}
\address{Mathematical Institute, University of Oxford, Andrew Wiles Building, Woodstock Road, Oxford, OX2 6GG, UK.}

\author{Andrew Thompson\fnref{fn1}\corref{cor1}}
\ead{thompson@maths.ox.ac.uk}
\address{Mathematical Institute, University of Oxford, Andrew Wiles Building, Woodstock Road, Oxford, OX2 6GG, UK.}

\cortext[cor1]{Corresponding author}
\fntext[fn1]{This publication is based on work partially supported by the EPSRC Centre For Doctoral Training in Industrially Focused Mathematical Modelling (EP/L015803/1) in collaboration with e-Therapeutics plc.}
\fntext[fn2]{In compliance with EPSRC's open access inititive, the data in this paper is available from https://doi.org/10.5061/dryad.51c59zw5r}	
	
\begin{abstract}
We propose a practical algorithm for low rank matrix completion for matrices with binary entries which obtains explicit binary factors and show it performs well at the recommender task on real world datasets. The algorithm, which we call TBMC (\emph{Tiling for Binary Matrix Completion}), gives interpretable output in the form of binary factors which represent a decomposition of the matrix into tiles. Our approach extends a popular algorithm from the data mining community, PROXIMUS, to missing data, applying the same recursive partitioning approach. The algorithm relies upon rank-one approximations of incomplete binary matrices, and we propose a linear programming (LP) approach for solving this subproblem. We also prove a $2$-approximation result for the LP approach which holds for any level of subsampling and for any subsampling pattern, and show that TBMC exactly solves the rank-\(k\) prediction task for a underlying block-diagonal tiling structure with geometrically decreasing tile sizes, providing the ratio between successive tiles is less than \(1/\sqrt{2}\). Our numerical experiments show that TBMC outperforms existing methods on recommender systems arising in the context of real datasets.
\end{abstract}	

	
\maketitle 


\section{Introduction}

\subsection{Matrix completion}

Matrix completion is an area of great mathematical interest and has numerous applications including recommender systems for e-commerce, bioactivity prediction and models of online content, such as the famous Netflix problem. 

The recommender problem in the e-commerce setting is the following: given a database where rows are users and column entries indicate user preferences for certain products, fill in the entries of the database so as to be able to recommend new products based on the preferences of other users. Typically these matrices are highly incomplete, since most users will only have experienced a small fraction of all available products~\cite{aggarwal2016recommender}. Similarly, in bioactivity prediction, compound-protein interaction (CPI) databases record bioactivity between potential drug compounds (rows, or users) and target proteins (columns, or products). Obtaining experimental evidence of all possible interactions is prohibitively expensive however, and therefore there are few known entries relative to the overall size of the database, see for example~\cite{liu2015improving}. 

Low rank approaches to matrix completion have been the focus of a great deal of theoretical and algorithmic exploration ever since the seminal work in~\cite{fazel2002matrix,candes2009exact}. In many cases, the problem is formulated as follows. Given a database \(\mathbf{A} \in \mathbb{R}^{m \times n}\), with observed entries in \(\Omega\), find a matrix \(\mathbf{X} \in \mathbb{R}^{m \times n}\) with minimal rank that matches the observed entries up to a given tolerance, i.e. which solves
\begin{equation}\label{eqn:matcompreal}
\min|| \mathcal{P}_{\Omega}(\mathbf{A} - \mathbf{X})||_2^2 \text{ s.t. }  \text{rank}(\mathbf{X} ) \le r
\end{equation}

where \(r\) is some small integer and \(\mathcal{P}_{\Omega}\) is the projection to the space of known entries, such that the error is evaluated only for the \((i,j) \in \Omega\). 

A variety of algorithms have been proposed to solve this non-convex problem, see~\cite{tanner2016low} and references therein. Another popular approach is to solve a convex relaxation involving the nuclear norm~\cite{candes2009exact}. Such algorithms have been applied successfully in a wide range of applications, including recommender systems~\cite{koren2009matrix}. 

\subsection{Binary matrix completion}

The recommender problem can be seen as a binary decision problem. When applied to a binary matrix, \(\mathbf{A} \in \mathbb{B}^{m \times n}\), the output of the matrix completion algorithms described above cannot be guaranteed to be binary, and a typical approach for matrix completion with binary data is to threshold the output from an algorithm for completion with real data~\cite{vinayak2014graph, ames2011nuclear}. As highlighted in \cite{xu2014jointly}, where databases follow model assumptions and recovery from solving is exact, factorisation reduces to assigning identical rows to the same cluster. However, this approach is not robust to the violation of assumptions and breaks down when recovery is not exact. This motivates the search for algorithms which explicitly seek binary solutions. 

The problem \eqref{eqn:matcompreal} can also be formulated as one of approximating \(\mathbf{X}\) as the product of factors, namely
\begin{equation}\label{eqn:matcompfactreal}
\min|| \mathcal{P}_{\Omega}(\mathbf{A} - \mathbf{UV}^T)||_2^2 \text{ s.t. }  \mathbf{U} \in \mathbb{R}^{m \times k}, \mathbf{V} \in \mathbb{R}^{n \times k}~.
\end{equation}
We are interested in this paper in approximating our database with binary factors, due to the greater interpretability of the output. To appreciate this, note that a binary factorisation 
$$\mathbf{UV^T}=\sum_{i=1}^k \mathbf{U}_{:,i} \mathbf{V}_{:,i}^T$$ 
decomposes a binary matrix into biclusters of rows and columns, often referred to in the itemset mining community as \emph{tiles}, see for example~\cite{geerts2004tiling}. This decomposition provides an explicit characterization of the row/column clusters that best explain the database. Low rank decompositions designed for real matrices, on the other hand, tend to be SVD-like and orthogonal, with negative entries, and it is less clear how to interpret these. Low rank matrix completion with nonnegativity constraints, for example in~\cite{xu2012alternating}, enforces non-negativity of factors, but does not address the issue of rounding errors induced by rounding non-integer values. Matrix completion algorithms were proposed in~\cite{xu2015cur,wang2017provably} which obtain decompositions in which one factor is composed of rows of the matrix, and is by consequence binary, although the other factor is generally non-integer. 

However, in the case where both factors are required to binary, research to date has largely focused on the case of \emph{binary matrix factorisation} (BMF) in which the matrix is fully observed. BMF has found applications in itemset mining of transactional and textual data~\cite{geerts2004tiling}, and also in the analysis of gene expression data~\cite{zhang2010binary}. In these applications, entries are in general fully observed, though possibly with noise. 

Various algorithms for BMF have been proposed. The problem can be formulated as an integer program, but the use of an integer programming solver is only tractable for very small problem sizes. An approach combining convex quadratic programming relaxations and binary thresholding was proposed in~\cite{zhang2010binary}. For the special case of rank-one approximation, linear programming relaxations were proposed in~\cite{shen2009mining,lu2011weighted, zhang2010binary}, the first two of which also both proved that the linear programming solution yields a $2$-approximation to the optimal objective. An approach to binary factorisation based on $k$-means clustering was considered in \cite{jiang2014clustering}. A local search heuristic capable of improving solutions obtained by other methods was proposed in~\cite{mirisaee2015improved}. Of particular relevance to this paper is the PROXIMUS algorithm, proposed in \cite{koyuturk2006nonorthogonal}, which identifies patterns within the data using recursive partitioning based on rank-one approximations. 

Closely related to BMF is the \emph{Boolean} matrix factorisation problem, in which \(\mathbf{UV}^T\) is replaced by the OR operator, \(\mathbf{U} \wedge  \mathbf{V}\), which allows the tiles to overlap. A number of algorithms also exist for Boolean factorisation, including the iterative ASSO algorithm~\cite{miettinen2008discrete} and integer programming~\cite{kovacs2018lowrank}.

\subsection{Our contribution and related work}

The novelty of this paper is the use of a partitioning approach to solve the problem of BMF with missing data, which can be formulated as follows. For \(\mathbf{A} \in \mathbb{B}^{m \times n}\), solve

\begin{equation}\label{eqn:rankktiling}
\min_{\begin{array}{c}\mathbf{U} \in \mathbb{B}^{m \times r} \\ \mathbf{V} \in \mathbb{B}^{n \times r} \end{array} }||\mathcal{P}_{\Omega}(\mathbf{A} - \mathbf{U}\mathbf{V}^T)||_2^2  ~.
\end{equation}

The contributions of this paper are as follows.\\

\begin{itemize}
	\item We propose TBMC (Tiling for Binary Matrix Completion), a low rank binary matrix completion algorithm (\Cref{sec:TBMC}). The algorithm extends the recursive partitioning approach of  \cite{koyuturk2006nonorthogonal} for BMF  by means of rank-one approximations.\\ 
	\item In particular, we propose using an LP rank-one approximation for missing data. We support this choice with a guarantee that it provides a $2$-approximation to the optimal objective value, showing that the reasoning of \cite{shen2009mining} holds in the missing data case (\Cref{sec:approximation}).\\
	
	\item We show that, under the assumption of a block diagonal model for our data, this algorithm correctly identifies tiles for a rank-\(k\) database with geometrically decreasing tiles, where the ratio between successive tiles is less than \(1/\sqrt{2}\).
	
	\item We show that our algorithm outperforms alternatives based on related heuristics and techniques for non-negative matrix completion and binary matrix completion, when tested on synthetic and real life data (\Cref{sec:numerical}). 
\end{itemize}
$ $\\
\indent The most closely related work we are aware of is the \emph{Spectral method} proposed in \cite{xu2014jointly} for bi-clustering databases with missing data. The authors use low rank completion to cluster neighbour rows and then redefine the column clusters based on cluster membership, in a similar fashion to $k$-means. We show that our algorithm outperforms the Spectral method when solving Problem~\eqref{eqn:rankktiling} for real world datasets.

The authors of ~\cite{yadava2012boolean} extend the ASSO algorithm for Boolean matrix factorisation to deal with missing data. However, it is worth pointing out that their setup, as well as solving a different problem, is primarily intended for only small amounts of missing data. Methods that involve linear programs can be straightforwardly extended to the missing data case, by evaluating the objective and enforcing constraints only for \((i,j) \in \Omega\), however we are not aware of any attempts to analyse their predictive power for the recommender problem. 

\subsection{Relation to graph-theoretic problems}

Viewing the database as the adjacency matrix of a graph, we can view~\eqref{eqn:rankktiling} as a clustering problem; in particular the problem is that of approximating the edge set of a partially observed graph as a union of bicliques. The factors \(\mathbf{U}\) and \(\mathbf{V}\) can be interpreted as indicating the rows and columns of these bicliques. {color{red}The authors of  \cite{vinayak2014graph} solve the related problem for the diagonal block model, expressing the database as a sum of a low rank matrix plus a sparse component to account for noise.  Our approach differs as we allow column overlap of clusters.} Bi-clustering approaches, in particular for clustering gene expression data, have been used to solve formulations similar to the BMF problem, with different assumptions about the underlying data, equivalent to considering different constraints on \(\mathbf{U}\) and \(\mathbf{V}\). However these have focused primarily on cluster recovery; our focus is on exploring the predictive power of different algorithms for solving \Cref{eqn:rankktiling}.

%
%
%
	\section{The TBMC algorithm} \label{sec:TBMC}

We present an algorithm for low rank binary matrix completion inspired by the partitioning approach of \cite{koyuturk2006nonorthogonal} that generates a binary factorisation based on recursive rank-one partitions.

\subsection{Partitioning} \label{sec:partition}

Our algorithm is inspired by the recursive partitioning approach of the PROXIMUS algorithm for binary matrix factorisation~\cite{koyuturk2006nonorthogonal}. For a given submatrix, \(\mathbf{B}\), the algorithm first calculates a binary rank-one approximation \(\{\mathbf{u},\mathbf{v}\}\). The matrix \(\mathbf{B}\) is then partitioned based upon the rows included in the tile \(\{\mathbf{u},\mathbf{v}\}\): if the \(i^{th}\) entry of \(\mathbf{u}\) is positive, then the \(i^{th}\) row  is included in \(\mathbf{B}_1\), else it is included in \(\mathbf{B}_0\). Both submatrices are added to the set of submatrices to be partitioned. 

The vector \(\mathbf{v}\) is used as the pattern vector for \(\mathbf{B}_1\). If the scaled Hamming distance, defined as \(H(\mathbf{v},A_{i,:})=\sum_{j:(i,j)\in \Omega} (v_{ij}\neq A_{ij})/\sum_{j:(i,j)\in \Omega} 1\), from \(\mathbf{v}\) to any of the rows in \(\mathbf{B}_1\) is less than some tolerance \(t\), for \(0<t<1\), then the tile \(\mathbf{u}\mathbf{v}^T\) is included in the decomposition and \(\mathbf{B}_1\) is not partitioned any further. In addition, if the rank-one approximation returns a row vector of \(1\)s, the tile is added to the decomposition. This process is repeated on successive submatrices until convergence. The algorithm terminates when no partitions remain.  An illustration of the splitting process can be seen in \Cref{fig:splitting}.

\begin{figure}
	\centering	
	\includegraphics[width=\linewidth]{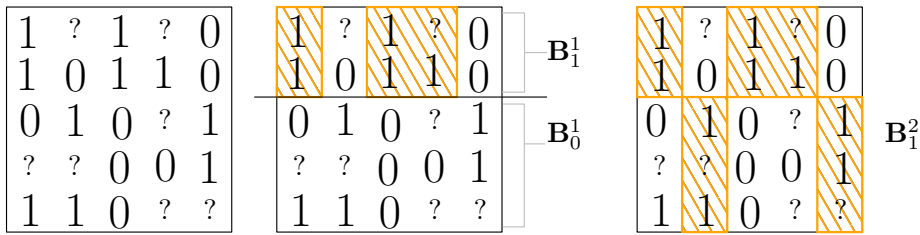}
	\caption{A row-wise partition is formed from a rank-one approximation (a tile) by splitting according to which rows are included in the tile.}
	\label{fig:splitting}
\end{figure}

\subsection{Rank-one approximation with missing data} \label{sec:rankone}

The partitioning method outlined above relies on rank-one approximations of the database. We outline two methods for this sub-problem.

\subsubsection{Approximate rank-one using a linear program}
Inspired by the work of \cite{shen2009mining}, we propose a linear programming relaxation of the rank-one approximation problem with missing data.

The rank-one problem can be formulated as an integer linear program for the case where no data is missing~\cite{mirisaee2015improved,kovacs2018lowrank}. The authors of \cite{shen2009mining} show that a linear program relaxation of a related problem has no more than twice the error of the integer solution. In \Cref{sec:approximation} we show that a similar result holds for the missing data case. 

We can formulate the best binary rank-one approximation problem as a linear program as follows:
\begin{equation}\label{bestlp1}
\begin{array}{ll}\renewcommand{\arraystretch}{1.2}
\min\limits_{ \tiny
	\begin{array}{l} \mathbf{u} \in \{0,1\}^{m} \\ \mathbf{v} \in \{0,1\}^{n}  \end{array}}  
|| \mathcal{P}_{\Omega}(\mathbf{A} - \mathbf{uv}^T)||_2^2 &=

\min\limits_{\tiny
	\begin{array}{l}
	\mathbf{u} \in \{0,1\}^m \\ 
	\mathbf{v} \in \{0,1\}^n
	\end{array}} 
\sum\limits_{(i,j) \in \Omega} (A_{ij}-u_i v_j)^2 \\

&= \min\limits_{\tiny
	\begin{array}{l}
	\mathbf{u} \in \{0,1\}^m \\ 
	\mathbf{v} \in \{0,1\}^n
	\end{array}} \sum\limits_{(i,j) \in \Omega} A_{ij}^2  - 2(A_{ij}-1)u_iv_j ~.

\end{array}
\end{equation}

Since \(A\) is fixed, (\ref{bestlp1}) is equivalent to solving  \begin{equation}\label{eqn:besttile}\max_{\tiny \begin{array}{l}
	\mathbf{u} \in \{0,1\}^m \\ 
	\mathbf{v} \in \{0,1\}^n 
	\end{array}} \sum\limits_{(i,j) : A_{ij}=1} u_iv_j - \sum\limits_{(i,j) : A_{ij}=0} u_iv_j.\end{equation}


We translate Problem (\ref{eqn:besttile}) to a linear program as follows: introducing dummy variables \(z_{ij}\) which serve as indicator variables for \(u_iv_j\) we can rewrite (\ref{eqn:besttile}) as

\begin{equation}\renewcommand{\arraystretch}{1.2}\label{eqn:IPbest}
\begin{array}{rl}
\max\limits_{\tiny
	\begin{array}{cl}
	\mathbf{u} \in \{0,1\}^m \\ 
	\mathbf{v} \in \{0,1\}^n \\
	\mathbf{z} \in \{0,1\}^{n \times m} \\
	\end{array}} &
\sum\limits_{(i,j) \in \Omega : A_{ij}=1} z_{ij}  -
\sum\limits_{(i,j) \in \Omega : A_{ij}=0} z_{ij} \\

\text{ s.t. }&
\left.
\begin{array}{ll}
 u_i+v_j -z_{ij} \leq 1  \\
2z_{ij} \leq u_i+v_j .
\end{array} \right\} (i,j) \in \Omega

\end{array}
\end{equation}

If we relax the integer constraints, then for optimal \(z_{ij}\), the second constraint will be satisfied as an equality for \(A_{ij}=1\), so we can drop the corresponding dummy variables and replace their value in the objective with \(\dfrac{1}{2}(u_i+v_j)\). Thus we consider the following formulation for approximating the solution to \Cref{eqn:IPbest}

\begin{equation}\renewcommand{\arraystretch}{1.2}\label{eqn:LPbest}
\begin{array}{rl}

\max\limits_{\tiny
	\begin{array}{cl}
	\mathbf{u} \in \mathbb{R}^m \\ 
	\mathbf{v} \in \mathbb{R}^n 
	\end{array}}
& \sum\limits_{(i,j) \in \Omega : A_{ij}=1} \dfrac{1}{2}(u_i+v_j)  -
\sum\limits_{(i,j) \in \Omega : A_{ij}=0} z_{ij} \\

\text{s.t.} 
& \left.
\begin{array}{l}
u_i+v_j -z_{ij} \leq 1  \\
0<z_{ij}<1
\end{array} \right\}(i,j) \in \Omega : A_{ij}=0

\\ &
\begin{array}{ll}
0<u_i,v_j<1&(i,j) \in \Omega\\
\end{array}

\end{array}
\end{equation}

The corresponding constraint matrix is totally unimodular (as it is a submatrix of the constraint matrix in \cite{shen2009mining}), and so solving will give integral values for \(\{u,v,z\}\). We can obtain the corresponding rank-one approximation as \(\mathbf{u}\mathbf{v}^T\). Note that \(z_{ij}\) is only defined for \((i,j):A_{ij}=0\). The remaining \(z\) values can be calculated as \(1/2(u_i+v_j)\) but are no longer guaranteed to be integers. \\

\subsubsection{Alternating minimisation}\label{subsec:altmin}

Alternating minimisation has been widely used for matrix completion, see for example \cite{jain2013low}. Given a fixed \(\mathbf{v}\) we want to find \(\mathbf{u}\) to minimise the approximation error
\begin{equation}\label{eqn:errorapprox}
    ||\mathcal{P}_{\Omega}(\mathbf{A}-\mathbf{u}\mathbf{v}^T)||_2^2= \sum_{i,j \in \Omega} A_{ij}^2- u_i(2A_{ij}v_j - v_j^2)
\end{equation}

where we have used the fact that \(\mathbf{u}\) and \(\mathbf{v}\) are binary. Writing
\begin{equation}\label{eqn:W_def}W_{ij}=\begin{cases}2 A_{ij}-1 & \text{ if } i,j \in \Omega \\
					0 & \text{ otherwise}, \end{cases}\end{equation}
we can write the right-hand side of (\ref{eqn:errorapprox}) as

\begin{equation}\label{eqn:errorapprox_missing}
\sum_i \sum_{j:(i,j) \in \Omega} A_{ij}^2-u_i W_{ij}v_j.
\end{equation}
Suppose we are at iteration $k$ and we are fixing $\mathbf{v}^k$. Then it follows from (\ref{eqn:errorapprox_missing}) that the updated binary \(\mathbf{u}^{k+1}\) is given by 

\begin{equation}\label{eqn:iter_threshold}u_i^{k+1} = \begin{cases} 1 & \text{ if } (\mathbf{W}\mathbf{v}^k)_i>0 \\	
							0& \text{ else.} \end{cases}\end{equation}

Then \(\mathbf{v}^{k+1}\) can be calculated in a similar way. Typically the process converges in only a few iterations.

The authors of \cite{koyuturk2006nonorthogonal} outline a number of heuristics for initialising the iterative scheme. In particular, if not initialised correctly, this method can lead to the empty tile, which is always sub-optimal. We propose using alternating minimisation as an optional post-processing step for the TBMC algorithm stated in Section \ref{sec:lpbest}.

\subsection{Statement of algorithm}\label{sec:lpbest}

Using the ideas of \Cref{sec:partition} and \Cref{sec:rankone} we now propose an algorithm for binary matrix completion for recommender systems. 

Starting with the full database, we calculate \(\{\mathbf{u^*}, \mathbf{v^*} \}\) the rank-one approximation obtained by solving Problem (\ref{eqn:LPbest}) for the current partition, then partition into the rows that are included in the tile and those that are not, and repeat for both sides of the partition. As in \cite{koyuturk2006nonorthogonal}, when the partitioning process generates rows that are all within a given radius of \(\mathbf{v}^*\) or when \(\mathbf{u}^*\) is the vector of all \(1\)s or all \(0\)s, then \(\{\mathbf{u^*}, \mathbf{v^*} \}\) is added to the factorisation. The algorithm is stated in detail in \Cref{alg:ours}. We could also consider using the alternating minimisation outlined in Section~\ref{subsec:altmin} to improve the approximation found by solving \eqref{eqn:LPbest}. We refer to this approach as TBMC-AM. In its current form, the algorithm only partitions by rows. We also considered a variant by which each tile identified by the algorithm was then partitioned in the same manner column-wise. It is possible to construct examples for which this would improve the accuracy but as we found no change in the output for any of our experiments, we speculate that these examples are rarely found in practise. 

\begin{minipage}[b]{0.8\linewidth}
	\begin{algorithm}[H] 
		\caption{Tiling for Binary Matrix Completion (TBMC)}\label{alg:ours}
		\begin{algorithmic}[1]	
			\State {Initialise} \(S=\left\{\mathbf{A}\right\}\), \(\mathbf{U}=\emptyset,\mathbf{V}=\emptyset\), \(i=1\)
			\While{$S\not=\emptyset$, \(k<k_{max}\)}
			\State select \(\mathbf{B} \in S\), set \(S \gets S \backslash \mathbf{B}\)

			\Procedure {find rank-one approximation for \(\mathbf{B}\):}{}
			\State Find \(\{\mathbf{u}, \mathbf{v}\}\) as the solution to \ref{eqn:LPbest}		
			\EndProcedure
			
			\Procedure {Partition \(\mathbf{B}\)}{}
			\State Define \(J= \left\{j : \mathbf{u}^k(j)=1\right\}\)
			\State Set \(\mathbf{\mathbf{B}}_1=\mathbf{B}(J,:)\), \(\mathbf{B}_0=\mathbf{B}(\overline{J},:)\)
			\EndProcedure
			\Procedure{Update}{}
			\If {\(\mathbf{B}_0  \neq \mathbf{0} \text{ and }  \mathbf{u}\neq \mathbf{0} \)} 
			\(S \gets B_0\)
			
			\EndIf
			
			\If{ \(\max_j {||\mathbf{B}_1(j,:)-v^k||< t} , \text{ or } \mathbf{u}  = \mathbf{1} \)
			
			}{ 
			
			\(~~~~~~~~U \gets \mathbf{u}^k\),\(V \gets \mathbf{v}^k\), \(k \rightarrow k+1\)}
			\Else
			\(~S \gets S \cup \mathbf{B}_1 \)

			\EndIf
			\EndProcedure
			\EndWhile
		\end{algorithmic}
	\end{algorithm}
\end{minipage}

\subsection{Algorithm complexity}

The most computationally demanding step in \(\Cref{alg:ours}\) is the solution of (\ref{eqn:LPbest}). To solve a non-degenerate LP with a simplex method, the worst case complexity is exponential in the number of variables (equal to \(N=n+m+|\{(i,j) : (i,j) \in \Omega, A_{ij}=0\}|\)) and the number of constraints (equal to \(M=2N+|\{(i,j) : (i,j) \in \Omega, A_{ij}=0\}|\)), since we have one constraint for each \(z_{ij}\), plus constraints that each variable be between \(0\) and \(1\). This upper bound arises because the method will terminate when it has checked every vertex of the feasible region defined by the problem. There are \(M+N\) possible vertices, so there are \(2^{M+N}\) possible intersections. However, since the constraint matrix of (\ref{eqn:LPbest}) is totally unimodular 
(see Section~\ref{sec:rankone}), a polynomial bound in $M$ and $N$ can be obtained. By Corollary 4.1 of \cite{kitahara2013bound}, solving (\ref{eqn:LPbest}) using Bland's rule~\cite{bland1977new} (or another method for avoiding cycles) will terminate in at most \(2N(M \log N+1)\) operations. 

The other steps in the algorithm have lower complexity. The alternating minimisation scheme requires one left and one right matrix vector multiplication per iteration, hence, since we impose a limit on the number of iterations, it takes \(\mathcal{O}(n+m)\) steps. Calculation of the error (step 11) requires at most \(n m\) operations.

In order to determine the complexity of the algorithm as a whole, we multiply the complexity of the rank-one case by \(m\), the worst case for the number of partitions that need to be checked.  So \(\Cref{alg:ours}\) requires at most \(\mathcal{O}(m N M log N)\) operations, i.e. polynomial in $m$ and $n$. 

We note that solving the problem directly involves checking \(k\) options for the membership of each row and column, i.e. \(k^{m+n}\) possibilities. Hence \(\Cref{alg:ours}\) offers a significant improvement over solving the problem directly.

	\section{Approximation bounds for the rank-1 subproblem}\label{sec:approximation}


In \Cref{sec:rankone}, we proposed a linear program method for solving the rank-one step of \eqref{alg:ours}. By closely following the approach in \cite{shen2009mining} for the case in which all the entries in \(\mathbf{A}\) are known, we show that the approximation error of the tile found by solving Problem \eqref{eqn:LPbest} is no more than twice the optimal. 

\begin{theorem}\label{thm:2approx}
	Denoting by \((\mathbf{u}^*,\mathbf{v}^*,\mathbf{z}^*)\) the optimal solution for Problem \eqref{eqn:LPbest} obtained using a simplex method, then the approximation error \(E=||\mathcal{P}_{\Omega}(\mathbf{A} - \mathbf{u}^*{\mathbf{v}^*}^T)||_2^2\) satisfies
	
	\begin{equation} \label{eqn:erroboundlpbest}
	E
	\leq 2 \min_{\tiny
		\begin{array}{l}
		\mathbf{u} \in \{0,1\}^m \\ 
		\mathbf{v} \in \{0,1\}^n 
		\end{array}}   
	||\mathcal{P}_{\Omega}(\mathbf{A}-\mathbf{u}\mathbf{v}^T)||_2^2
	\end{equation}
\end{theorem} 
\begin{proof}
	
	The minimal error for a single tile approximation is given by 
	
	\begin{equation}\renewcommand{\arraystretch}{2}
	\begin{array}{ll}
	\min\limits_{\tiny
		\begin{array}{l}
		\mathbf{u} \in \{0,1\}^m \\ 
		\mathbf{v} \in \{0,1\}^n 
		\end{array}} &
	||\mathcal{P}_{\Omega}(\mathbf{A} - \mathbf{u}\mathbf{v}^T)||_2^2
	
	\\
&= \min\limits_{\tiny
		\begin{array}{l}
		\mathbf{u} \in \{0,1\}^m \\ 
		\mathbf{v} \in \{0,1\}^n 
		\end{array}} 
	\sum\limits_{(i,j) \in \Omega}(A_{ij}-u_iv_j)^2 \\ 
&= \sum\limits_{(i,j) \in \Omega} A_{ij}^2 - 
	\max\limits_{\tiny
		\begin{array}{l}
		\mathbf{u} \in \{0,1\}^m \\ 
		\mathbf{v} \in \{0,1\}^n 
		\end{array}} 
	\left(\sum\limits_{(i,j):A_{ij}=1} u_iv_j -  \sum\limits_{(i,j):A_{ij}=0} u_iv_j\right) \\
&\geq \sum\limits_{(i,j) \in \Omega} A_{ij}^2 - \left(\dfrac{1}{2} \sum\limits_{(i,j):A_{ij}=1} (u^*_i+v^*_j) - \sum\limits_{(i,j):A_{ij}=0} z^*_{ij}\right)
	\end{array}
	\end{equation}
	
	where the inequality is a result of the fact that the optimal solution for Problem \eqref{eqn:IPbest} is bounded above by the objective for Problem \eqref{eqn:LPbest} since relaxing constraints does not decrease the value at optimality.

	For \((i,j)\) such that \(A_{ij}=0\), we know that \(u_i\), \(v_j\) and \(z_{ij}\) are integers. Since the objective of (\ref{eqn:LPbest}) is to minimise \(z^*_{ij}\), which is bounded below by \(u^*_i+v^*_j -1\), we have that \(z^*_{ij}=1\) if and only if \(u^*_i+v^*_j=2\).  Thus we can split the summation terms on the right hand side of the inequality to obtain
	
	\begin{equation}\label{eqn:errorboundblah}
LHS \geq \sum\limits_{\tiny (i,j) \in \Omega} A_{ij}^2 - \sum\limits_{\tiny\begin{array}{ll} (i,j):&A_{ij}=1\\&u_i+v_j=2 \end{array}} 1 -
	\sum\limits_{\tiny\begin{array}{ll}\tiny (i,j):&A_{ij}=1\\&u_i+v_j=1 \end{array}} \dfrac{1}{2} +
	\sum\limits_{\tiny\begin{array}{ll}\tiny (i,j):&A_{ij}=0\\&u_i+v_j=2 \end{array}} 1
	\end{equation}
	
	Since \(A\) is binary, we have that 
	\begin{equation}
	\sum\limits_{\tiny(i,j) \in \Omega} A_{ij}^2 \geq \sum\limits_{\tiny\begin{array}{ll} (i,j):&A_{ij}=1\\&u_i+v_j=1 \end{array}}1 + \sum\limits_{\tiny\begin{array}{ll} (i,j):&A_{ij}=1\\&u_i+v_j=2 \end{array}}1
	\end{equation}
	
	which we can use to replace the sum in \eqref{eqn:errorboundblah} corresponding to values where exactly one of \(u_i\) and \(v_j\) is 1 to derive
	
	\begin{equation}\label{eqn:pequality}\renewcommand{\arraystretch}{2}\renewcommand{\arraycolsep}{1pt}
		\begin{array}{rl}
		\min\limits_{\tiny
			\begin{array}{l}
			\mathbf{u} \in \{0,1\}^m \\ 
			\mathbf{v} \in \{0,1\}^n 
			\end{array}} 
		||\mathcal{P}_{\Omega}(\mathbf{A} - \mathbf{u}\mathbf{v}^T)||_2^2 \geq& \dfrac{1}{2}\left( \sum A_{ij}^2 - \sum\limits_{\tiny\renewcommand{\arraystretch}{1}\begin{array}{ll}\tiny (i,j):&A_{ij}=1\\&u_i+v_j=2 \end{array}} 1\right) + \sum\limits_{\tiny\renewcommand{\arraystretch}{1}\begin{array}{ll} (i,j):&A_{ij}=0\\&u_i+v_j=2 \end{array}} 1  \\
	
	\geq& \dfrac{1}{2}\left( \sum A_{ij}^2 - \sum\limits_{\tiny\renewcommand{\arraystretch}{1}\begin{array}{ll}\tiny (i,j):&A_{ij}=1\\&u_i+v_j=2 \end{array}} 1 + \sum\limits_{\tiny\renewcommand{\arraystretch}{1}\begin{array}{ll}\tiny (i,j):&A_{ij}=0\\&u_i+v_j=2 \end{array}} 1\right)  \\
	\end{array}
	\end{equation}
	Now since \(2A_{ij}-1\) is equal to  \(1\) if \(A_{ij}\) is positive and \(-1\) if \(A_{ij}\) is zero we can rewrite (\ref{eqn:pequality}) as 
	
	\begin{equation}
	\begin{array}{lll}
			\min\limits_{\tiny
				\begin{array}{l}
				\mathbf{u} \in \{0,1\}^m \\ 
				\mathbf{v} \in \{0,1\}^n 
				\end{array}} 
			||\mathcal{P}_{\Omega}(\mathbf{A} - \mathbf{u}\mathbf{v}^T)||_2^2 \geq
	&=&\dfrac{1}{2} \sum\limits_{(i,j)\in \Omega}  \left(A_{ij}^2 - \left(2A_{ij}-1\right)u_iv_j \right)\\
	&=& \dfrac{1}{2} \sum\limits_{(i,j)\in \Omega} \left(A_{ij} -u_iv_j \right)^2 = \dfrac{1}{2}E~.
	\end{array}
	\end{equation}
	This gives us our bound.
\end{proof}
Note that in the case where the true best solution has zero error, the LP relaxation will also have zero error, and therefore for a database consisting of a single planted tile, solving (\ref{eqn:LPbest}) will recover this tile provided that, for each row, the sampling operator sees a positive and a negative entry for each row or column. 

\subsection{Verifying \Cref{thm:2approx}}\label{subsec:rank1}

We first perform experiments to illustrate numerically that the 2-approximation result given in \eqref{eqn:erroboundlpbest} is not violated.

We evaluate performance relative to the optimal solution, which we obtain by solving (\ref{eqn:IPbest}) directly as an integer program. We generate binary matrices with (a) a single planted tile and (b) 3 planted tiles, with \(\tau n\) positive entries per row in each tile. We simulate noisy data by randomly flipping a fraction \(\epsilon\) of entries and remove a proportion \(1-\rho\) of entries. We set \(\tau=0.7\) and \(\epsilon=0.03, \rho=0.7\). Note that choice of parameters is to provide an illustrative example, as we observe similar behaviour for other parameter configurations. To make it tractable to solve the problem directly, we consider databases of size \(100\times100\) for the single tile case and \(10\times10\) for the \(3\) tile case, since the latter is more computationally intensive to solve. In both cases, we calculate the ratio, \(\mathcal{R}\), of squared \(l_2\) approximation error to the optimal squared \(l_2\) error. The mean value of \(\mathcal{R}\) and the proportion of cases for which \(\mathcal{R}\) is greater than \(1\) is recorded in \Cref{tab:2approx}.

We solve the LP relaxation using the simplex method implemented by MATLAB's linprog solver \cite{dantzig1955generalized}. In addition, we compare against other methods, \emph{average, partition}, and \emph{nmf} for generating a rank-1 approximation, further details of which can be found in \Cref{sec:numerical}. We observe that while the other methods all violate the 2-approximation bound for the 3-tile model, the LP does not.

\begin{minipage}{0.45\linewidth}
\begin{tabular}{c|c|c|cllll}

\cline{2-3}
& \cellcolor[HTML]{EFEFEF}Mean \(\mathcal{R}\)
& \cellcolor[HTML]{EFEFEF}\(P_0\)
& \\ 

\cline{1-3}
\multicolumn{1}{|c|}{\cellcolor[HTML]{EFEFEF}\begin{tabular}[c]{@{}c@{}}LP\\ (AM)\end{tabular}}        
& \begin{tabular}[c]{@{}c@{}}1.0\\\bf{(1.0)}\end{tabular} 
& \textbf{\begin{tabular}[c]{@{}c@{}}0\\(0)\end{tabular}}& \\

 \cline{1-3}
\multicolumn{1}{|c|}{\cellcolor[HTML]{EFEFEF}\begin{tabular}[c]{@{}c@{}}NMF\\ (AM)\end{tabular}}       
& \begin{tabular}[c]{@{}c@{}}0.0\\ \bf{(0.0)}\end{tabular} & \begin{tabular}[c]{@{}c@{}}0.0\\ \bf{(0.0)}\end{tabular} & \\

 \cline{1-3}
\multicolumn{1}{|c|}{\cellcolor[HTML]{EFEFEF}\begin{tabular}[c]{@{}c@{}}Partition\\ (AM)\end{tabular}} 
& \begin{tabular}[c]{@{}c@{}}2.42\\ (1.26)\end{tabular} & \begin{tabular}[c]{@{}c@{}}0.87\\ \bf{(0.0)}\end{tabular} &\\ 

\cline{1-3}
\multicolumn{1}{|c|}{\cellcolor[HTML]{EFEFEF}\begin{tabular}[c]{@{}c@{}}Avg\\ (AM)\end{tabular}}
& \begin{tabular}[c]{@{}c@{}}2.56\\ (1.32)\end{tabular} & \begin{tabular}[c]{@{}c@{}}0.62\\ (0.0)\end{tabular} & \\ 
\cline{1-3}
\end{tabular}
\captionof{table}{Mean value of \(\mathcal{R}\) and \(P_0\), the proportion of cases test cases for which \(\mathcal{R}>2\), for matrices with a single planted tile, \(m=100\).}
\end{minipage}\hfill
\begin{minipage}{0.45\linewidth}
\begin{tabular}{c|c|c|cllll}
\cline{2-3}
                                                                       & \cellcolor[HTML]{EFEFEF}Mean \(\mathcal{R}\) & \cellcolor[HTML]{EFEFEF}\(P_0\)                          &           &           &  &  &  \\ \cline{1-3}
\multicolumn{1}{|c|}{\cellcolor[HTML]{EFEFEF}\begin{tabular}[c]{@{}c@{}}LP\\ (AM)\end{tabular}}        & \begin{tabular}[c]{@{}c@{}}1.06\\ \(\mathbf{(1.04)}\)\end{tabular} & \textbf{\begin{tabular}[c]{@{}c@{}}0\\ (0)\end{tabular}} & \textbf{} & \textbf{} &  &  &  \\ 

\cline{1-3}
\multicolumn{1}{|c|}{\cellcolor[HTML]{EFEFEF}\begin{tabular}[c]{@{}c@{}}NMF\\ (AM)\end{tabular}}       & \begin{tabular}[c]{@{}c@{}}2.46\\ (1.21)\end{tabular} & \begin{tabular}[c]{@{}c@{}}0.43\\ (0.01)\end{tabular}    &           &           &  &  &  \\ 

\cline{1-3}
\multicolumn{1}{|c|}{\cellcolor[HTML]{EFEFEF}\begin{tabular}[c]{@{}c@{}}Partition\\ (AM)\end{tabular}} & \begin{tabular}[c]{@{}c@{}}2.94\\ (1.37)\end{tabular} & \begin{tabular}[c]{@{}c@{}}0.58\\ (0.12)\end{tabular}    &           &           &  &  &  \\ \cline{1-3}
\multicolumn{1}{|c|}{\cellcolor[HTML]{EFEFEF}\begin{tabular}[c]{@{}c@{}}Avg\\ (AM)\end{tabular}}       & \begin{tabular}[c]{@{}c@{}}2.96\\ (1.21\end{tabular} & \begin{tabular}[c]{@{}c@{}}0.74\\ (0.06)\end{tabular}     &           &           &  &  &  \\ \cline{1-3}
\end{tabular}
\label{tab:2approx}
\captionof{table}{Mean value of \(\mathcal{R}\) and proportion of cases test cases for which \(\mathcal{R}>2\), for matrices with 3 planted tiles with density \(\tau=0.7\) with \(m=10\).}
\end{minipage}



    \section{Recovery guarantees for TBMC}\label{sec:rank_k}

The results of \Cref{sec:approximation} hold for any binary matrix and give guarantees for solving the rank-1 problem. We now explore the conditions for which \Cref{alg:ours} will recover a database generated according to a planted tile model. We consider the following model for the underlying data: let  \(\mathbf{A} \in \mathbb{B}^{m\times m}\) have a symmetric block diagonal structure, with the size of the \(l^{th}\) tile equal to \((m\tau_l)^2\) where \(\tau_{l+1}\leq \tau_l\) as illustrated in \Cref{fig:rank_k_block_diagonal}. 

\begin{minipage}[t]{0.45\linewidth}
    \centering
    \includegraphics[width=0.6\linewidth]{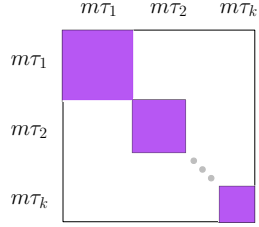}
    \captionof{figure}{We consider recovery guarantees for case of a block diagonal model with decreasing tile size.}
    \label{fig:rank_k_block_diagonal}
\end{minipage}\hfill
\begin{minipage}[t]{0.45\linewidth}
    \centering
    \includegraphics[width=0.7\linewidth]{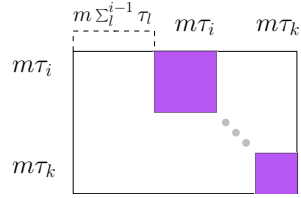}
    \captionof{figure}{On the \(i^{th}\) iteration of \Cref{alg:ours}, empty columns can be ignored, as setting \(\beta_j=1\) for \(j<i\) will never increase the objective.}
    \label{fig:my_label}
\end{minipage}

The heart of the proof is in showing that, in the case of no missing data, solving Problem (\ref{eqn:LPbest}) recovers the largest tile. We then show that for the case of no missing data, if the algorithm recovers the first \(i\) tiles then it will recover the \((i+1)^{th}\) tile and conclude. Finally, we show numerically that the performance for the case of missing data is in close agreement with the theoretical result for no missing data. We define \(T_l\) to be the set of indices in the \(l^{th}\) tile and 
$$S_{l,l'}=\{(i,j):i\in T_l,j\in T_{l'}\}.$$
We rely on the following observations:
	\begin{lemma}
 We can rewrite the objective of (\ref{eqn:LPbest}) as
	
	\begin{equation}\label{eqn:LPbest_rewrite}\renewcommand{\arraystretch}{0.5}
	\begin{array}{lll}
f(\mathbf{u},\mathbf{v})&=&\displaystyle	 \sum\limits_{l} \left(	\sum\limits_{\begin{array}{c}(i,j)\in S_{l,l}\\u_i+v_j=1\end{array}}
	\frac{1}{2} +		\sum\limits_{\begin{array}{c}(i,j)\in S_{l,l}\\u_i+v_j=2\end{array}}1 \right)\\
	&&\\
&& -\displaystyle 
	\sum\limits_{l\neq l'}	\sum\limits_{\begin{array}{c}(i,j)\in S_{l,l'}\\u_i+v_j=2\end{array}} 1.

	    \end{array}
	\end{equation}

	\end{lemma}
	
	\begin{proof}
	
The constraints of (\ref{eqn:LPbest}) enforce the condition that \(z_{ij}=u_iv_j\). Hence, if we split the second term of the objective of Problem (\ref{eqn:LPbest}) by the values of \(u_i+v_j\) the only terms that will have a contribution are the ones for which \(u_i+v_j=2\). The rest follows by splitting the summation terms according to the values of \(u_i+v_j\). 

\end{proof}

\begin{lemma}\label{lem:alpha_beta}
Let \(\alpha_l,\beta_l\) be the proportion of rows and columns of tile \(l\) included in a feasible solution to (\ref{eqn:LPbest}). Then we can write (\ref{eqn:LPbest_rewrite}) as

\begin{equation}\label{eqn:proportion_objective}
\begin{array}{ll}

    f(\mathbf{\alpha},\mathbf{\beta})= 
  &m^2\displaystyle \sum\limits_l \tau_l^2 \left\{\frac{1}{2} \left[ 
    \alpha_l(1-\beta_l)+(1-\alpha_l)\beta_l \right] + \alpha_l\beta_l\right\}\\
    
    & - \displaystyle m^2 \sum\limits_{l=1}^{k}\sum\limits_{\centering \begin{array}{l} _{l'=1} \\_{l'\neq l}\end{array}}^k \alpha_{l'}\beta_l\tau_{l'}\tau_l .

    \end{array}
\end{equation}

 In addition, there is an optimal solution for which \(\alpha_l,\beta_l\in\{0,1\}\) for each tile. 
\end{lemma}
\begin{proof}This follows from \Cref{eqn:LPbest_rewrite}. The objective can be split into sums by row; rows in the same tile will have identical contribution to the objective, leading to the formulation in (\ref{eqn:proportion_objective}). Since this expression is linear in the \(\alpha_l\) for each \(l\), we can conclude that there must be an optimal solution for either \(\alpha_l=0\) or \(\alpha_l=1\); to determine which requires determining  the sign of the coefficient of each \(\alpha_l\): if it is positive then \(\alpha_l=0\) is optimal, and if it is negative then \(\alpha_l=1\) is optimal. The same reasoning applies to each \(\beta_l\).

\end{proof}

We analyse tile recovery for the case where all entries are known, obtaining a phase transition for recovery in terms of relative tile sizes, and then show that, in practice, we get agreement with this phase transition for sub-sampled matrices, provided the tiles are large enough. 

\begin{theorem}
           For a binary matrix \(\mathbf{A}\) with a symmetric block diagonal structure with \(k\) tiles, having size \(m\tau_l^2\) where \(\tau_{l+1}=a\tau_l\), for  \(a\leq1/\sqrt{2}\),
           \Cref{alg:ours} recovers the tiles exactly in the case where all entries are known (in this case \Cref{alg:ours} reduces to the PROXIMUMS algorithm \cite{koyuturk2006nonorthogonal}).%
\end{theorem}

\begin{proof}
Using Lemma~\ref{lem:alpha_beta}, we may consider four cases depending on the integer values of \(\alpha_1,\beta_1\). 

Case (i): For \(\alpha_1=\beta_1=1\), we calculate

\begin{equation}
\begin{array}{lll}
\dfrac{1}{m^2} \dfrac{\partial f}{\partial \alpha_j}&= \frac{1}{2}\tau_j^2-\tau_j\tau_1- \sum\limits_{\begin{array}{l}_{l\neq j} \\ _{l >1}\end{array}} \tau_j\tau_l~\text{ for }j>1 \\
&\leq \frac{1}{2}\tau_j^2-\tau_j\tau_1
\end{array}
\end{equation}

This is strictly negative, since \(\tau_1>\tau_j\), hence \(\alpha_j=0\) for all \(j>1\). Similarly, \(\beta_{j}=0\) for all \(j>1\). The corresponding objective value is \(m^2\tau_1^2\).

Cases (ii) and (iii): For \(\alpha_1=1\), \(\beta_1=0\), 

\begin{equation}
\dfrac{1}{m^2} \dfrac{\partial f}{\partial \beta_j}\leq \frac{1}{2}\tau_j^2-\tau_j\tau_1~\text{ for }j>1~.
\end{equation}

Hence \(\beta_j=0\) for all \(j>1\). We can then substitute this to find

\begin{equation}
\begin{array}{ll}
f(\mathbf{\alpha},\mathbf{\beta})=& \frac{1}{2} m^2\tau_1^2\left(\tau_1^2+\displaystyle\sum_{l>1}\alpha_l\tau_l^2\right).
\end{array}
\end{equation}

So for optimality in this case, \(\alpha_l=1,\beta_l=0\) for all \(l\). By symmetry, for \(\alpha_1=0\), \(\beta_1=1\), optimality in this case is obtained for \(\alpha_l=0,\beta_l=1\) for all \(l\). In both cases, the objective value is \(f=\frac{1}{2}m^2\sum_l\tau_l^2\).

Case (iv): For \(\alpha_1=0\), \(\beta_1=0\), the objective is at most \(m^2 \sum_{l>1} \tau_l^2/2\). To see this, consider that the \(l^{th}\) term of the first summation is increasing in \(\alpha_l\) and \(\beta_l\), so obtains its maximum for \(\alpha_l=\beta_l=1\), and the second summation is always negative. 

We can now compare the objectives for each case. Case (i) is optimal over the other cases, provided

\begin{equation}
\sum_{l>1}\tau_1^2 >\frac{1}{2}\sum_l\tau_l^2,
\end{equation}

which is true if and only if

\begin{equation}\label{eqn:condtau}
\tau_1^2>\sum\limits_{l>1}\tau_l^2~.
\end{equation}

Since this is a strict inequality, the optimal solution will be unique. 

Note that on subsequent iterations of step 2 of~\Cref{alg:ours}, the input matrix is no longer square, but has dropped out the rows corresponding to the largest \(i\) tiles.  The resulting matrix contains the remaining \(k-i\) tiles and empty columns corresponding to the dropped tiles. Hence we are no longer tracking the \(\alpha_j\) for \(j\leq i\). Therefore, we rewrite the objective as

\begin{equation}
\begin{array}{ll} \displaystyle
f(\mathbf{\alpha},\mathbf{\beta})= &m^2\sum\limits_{l=i+1}^{k} \tau_l^2 \left\{\frac{1}{2} \left[ 
\alpha_l(1-\beta_l)+(1-\alpha_l)\beta_l \right] + (1-\alpha_l)(1-\beta_l)\right\}\\%
&\displaystyle - m^2\sum\limits_{l=i+1}^k\sum\limits_{\renewcommand{\arraystretch}{0.5}\begin{array}{l}	_{l'=i+1}\\_{l'\neq l}
\end{array}}^k\alpha_{l'}\beta_l\tau_{l'}\tau_l - 
m^2\sum\limits_{l=1}^i\sum\limits_{\renewcommand{\arraystretch}{0.5} \begin{array}{l}	_{l'=i+1}\\_{l'\neq l}
	\end{array}}^k\alpha_{l'}\beta_l\tau_{l'}\tau_l 
\end{array}
\end{equation}

where we have dropped out the terms corresponding to the dropped rows, and isolated terms corresponding to the empty columns (\(l\leq i \)). In this case, the coefficient of \(\beta_j\) for \(j\leq i\) will be non-positive, and negative provided at least one of the \(\alpha_j\) is positive for \(j>i\). This means we can consider optimality for the reduced matrix formed by dropping out the columns corresponding to the first \(i\) tiles and therefore if
\begin{equation}\label{eqn:condtau}
\tau_i^2>\sum\limits_{l>i}\tau_l^2~.
\end{equation}

we have that at each iteration, the algorithm will successfully identify the right tile, and therefore \ref{alg:ours} will output the exact tiling after \(k\) iterations. 

In particular, if we consider the geometric model for tile size, \(\tau_{l+1}=a\tau_l\), then  imposing (\ref{eqn:condtau}) leads to the condition 

\begin{equation}\renewcommand{\arraystretch}{1.5}
\begin{array}{lll}
    1&>& \sum_{l=1}^{k-i} a^{2l}=a^2\sum_{l=1}^{k-i} a^{2(l-1)} =a^2\dfrac{1-a^{2(k-i)}}{1-a^2}
    \end{array}
\end{equation}

for all $0\le i\le k-1$, from which we obtain 

\begin{equation}\label{eqn:conda}
    2a^2 -a^{2(k-i+1)}<1
\end{equation}

In particular, for \(a\leq1/\sqrt{2}\) this requirement is satisfied for all \(i\).

\end{proof}

\subsection{Observed behaviour with missing entries}

Now suppose that entries are erased independently with probability $0\le\rho<1$. We are interested in whether the phase transition identified above can be extended to this case. For a given realisation, we can no longer consider each row as identical, hence the theoretical result in this section does not precisely extend. However, our numerical observations suggest that we do observe a behaviour which is a close approximation, provided the database size is large enough.

We generate symmetric block-diagonal matrices with decreasing tile size according to the model in \Cref{sec:rank_k}, setting \(k=4\), for \(a=0.1:1\) and \(\rho=0.1\) to \(0.9\). For 100 random trials, we calculate the proportion of matrices recovered exactly from \Cref{alg:ours} and plot the results in the top row of  \Cref{fig:error_bin}. We also plot, in the bottom row of \Cref{fig:error_bin}, the proportion for which \Cref{alg:ours} is able to recover all but \(3\%\) of the entries; accounting for the difficulty of recovery of the smallest tiles. According to (\ref{eqn:conda}), we would expect to see recovery for \(a\leq 0.72\). In both cases, we see that there is agreement with our bound on \(a\), and that as we increase the size of our database, the value of maximum \(\rho\) for which the algorithm recovers all of the tiles also increases. The small gap between the dashed line indicating \(a=0.72\) and the line of failure to recover is an artefact of discretisation error for smaller databases; it is not possible to generate tiles that are exactly a factor of \(0.72\) smaller than the previous, since a row is either included or not.

\begin{minipage}{\textwidth}
\centering
\includegraphics[width=0.8\linewidth]{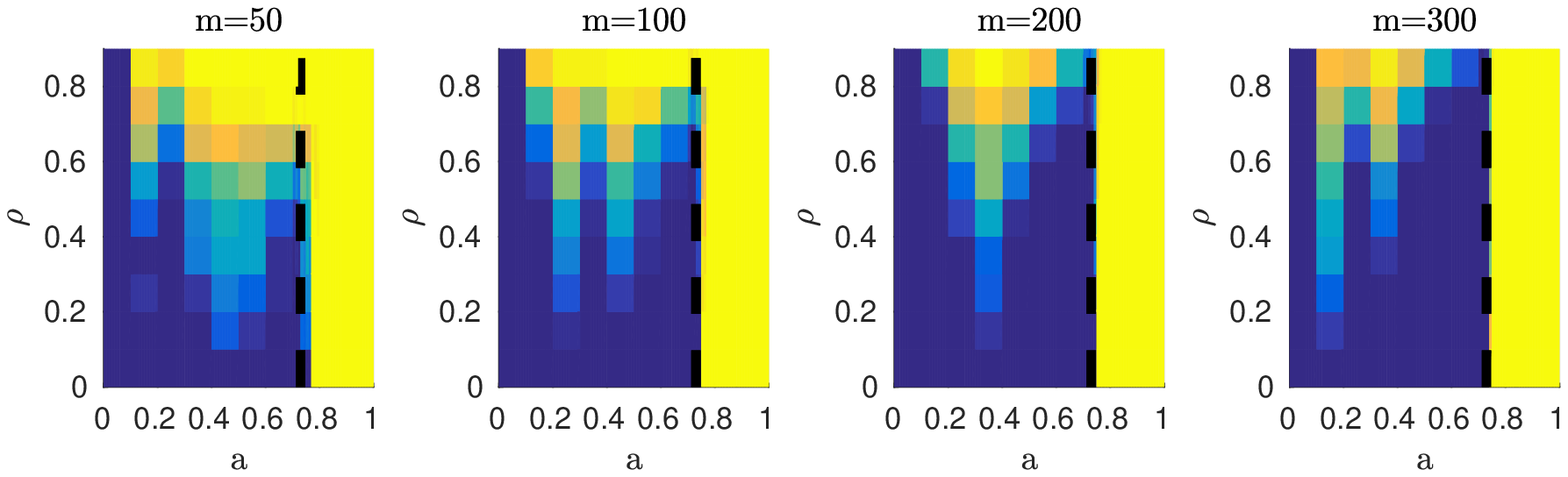}
\includegraphics[width=0.8\linewidth]{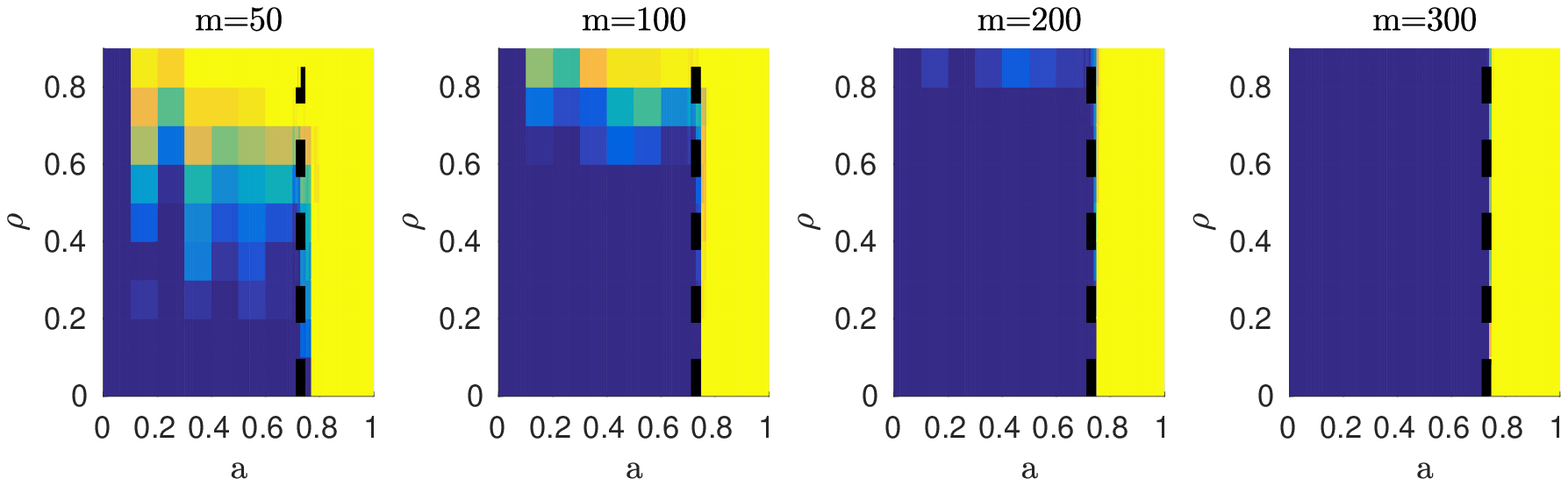}
\captionof{figure}{Proportion of matrices, generated according to \Cref{sec:rank_k}, which are (top) recovered by TBMC and (bottom) for which 97\% of the entries are recovered (bottom). Dark blue regions are fully recovered. The black dashed line indicate the theoretical bound on \(a\). }
\label{fig:error_bin}
\end{minipage}

	\section{Numerical Results} \label{sec:numerical}

We next demonstrate that the partitioning approach, and in particular our TBMC algorithm is a practical method for generating interpretable rules for inferring missing information in real data sets (recommender systems), that outperforms state of the art alternative methods for binary matrix completion.

\subsection{Other algorithms}

We compare \(\Cref{alg:ours}\) to other methods for rank-\(k\) approximation: two alternative heuristics for partitioning similar to those used in \cite{koyuturk2006nonorthogonal}; NMF with missing values implemented using the NMF MATLAB Toolbox \cite{qi2009non}, based on the multiplicative scheme of \cite{lee1999learning} with zero weighting of missing entries; and a method that follows a \(k\)-means approach to cluster the observed data based on projections onto the \(k\) largest singular vectors of the observed data. In particular we consider \begin{itemize}
	\item \emph{`average'}: setting \(v_j=0\) where the \(j^{th}\) column has more than half of its entries as positive and \(u_i=1\) if any of the entries of the \(i^{th}\) row have a positive entry in one of those columns ;
	\item \emph{`partition'}: selecting a column at random to be \(\mathbf{u}\) and calculating \(\mathbf{v}\) as the average of all the rows that have positive entries in that column, binarised with a threshold of \(1/2\);
	\item \emph{`NMF'}: a non-negative rank-one factorisation is obtained using a multiplicative update scheme with missing data and the factors are binarised using a threshold of \(0.5\). We use the approach outlined in \cite{zhang2010binary} to normalize the factors such that the 0.5 threshold is appropriate.
	\item \emph{`Spectral'}: We obtain a rank-one approximation by implementing the method outlined in \cite{xu2014jointly}, which follows a k-means approach. Briefly, the factors are initially generated using a projection onto the first \(k\) singular values; viewing the factorisation as a clustering,  the footprints in \(\mathbf{V}_i\) are updated using the clustering given by \({U}_i\), before reassigning rows to the cluster whose footprint most closely matches their own.
\end{itemize}

For TBMC, \emph{average} and \emph{partition} we also consider using an alternating minimisation step, outlined in \Cref{subsec:altmin} to improve the rank-one approximations. We refer to these approaches as TBMC-AM, \emph{average}-AM and \emph{partition}-AM.

\subsection{Real world datasets}

We consider performance on a  series of real world recommender systems. It is worth pointing out that whilst many of the relevant datasets have non-binary entries, often categorical ratings for example, the aim of the recommender problem is to decide whether or a particular entry will be a positive interaction or not (i.e. whether a user will like a particular film, or a drug will bind to a particular target). Hence it is necessary to make decisions about how these datasets are binarised. 

In all of the datasets we consider, we find optimum performance for partition based methods, and in particular  for four of the five data sets that we investigate, we find that TBMC outperforms other state of the art methods upon this task. 

We consider the performance of rank-$k$ binary matrix factorisation on the following datasets.

\begin{itemize}
\item ChEMBL:  Data of protein-ligand interactions, obtained from ChEMBL \cite{bento2014chembl}, version 21. Assay values for measurement types POTENCY, IC50, KI, MIC, EC50, KD, AC50 were binarised using a threshold of 2\(\mu\)M to obtain positive and negative activity categories, where negative represents no discernible effect on protein behaviour. The resulting matrix was then filtered by counting the positive and negative interactions, ranking the compounds based on the total number of interactions and taking subsets from the top 50000 ligands. 

 \item ML: MovieLens $100$k \cite{MovieLens}, which contains $100,000$ ratings from $943$ users across $1682$ films. We threshold to consider only $4$ or $5$ star ratings as positive.
 
  \item RC: customer reviews from $138$ customers for $130$ restaurants \cite{RC}. Since ratings are in $\{1,2,3\}$, we map to binary data by converting $3$ star ratings to one, and converting $1$ or $2$ star ratings to zero. 
 
 \item ALL-AML (GE): gene expression data \cite{golub1999molecular}, containing information from human tumour samples of different leukemia types. This dataset has been widely used for cancer classification, and for the evaluation of BMF techniques in the case of no missing data \cite{brunet2004metagenes}. We normalize the data by column. 
 
 \item Netflix: a subset of data obtained from Kaggle~\cite{lao2019Netflix} from the Netflix prize data~\cite{netflix} of ratings by users on different movies. The sparsely reviewed movies (\(<2000\) ratings) and sparsely active users (\(<52\) reviews) are removed. The resulting dataset reflects the \(70^{th}\) percentile for density of reviews. In order to apply the algorithms below, we consider a random subset of 1000 users and 1000 movies. We threshold to consider only $4$ or $5$ star ratings as positive.

\end{itemize}
We compare the performance of TBMC and TBMC-AM against NMF and the Spectral method (see Section~\ref{subsec:rank1}). The rank parameter was optimised for NMF, with $5$ being found to be a good choice. The maximum rank was set to be $5$ in the Spectral method, although the method was found to often terminate before five tiles were found. Although increasing the rank initially leads to an improve in approximation accuracy, none of these methods are guaranteed to find a zero error approximation simply by increasing the rank. In addition, there is a risk of over-fitting for a rank parameter that is too high.  We also compare with a variant of the recursive partitioning approach of TBMC in which the LP-based rank-one approximation step is replaced by the \emph{average} and \emph{partition} heuristics (see Section~\ref{subsec:rank1}), again with and without alternating minimisation. 

\subsection{Performance of TBMC for rank-k decomposition}

We consider a $70$/$30$ split between known training entries and unseen test entries (\(\rho=0.7\)) and set \(t=0.05\), although we find the impact of changing \(t\) has little impact below \(0.1\). Note that MovieLens, RC, ChEMBL and Netflix have a low density of known entries \(\rho_{D}\). In all cases we record the percentage approximation error $\mathcal{P}$. We give approximation errors upon both the test set, and also upon the training set, averaged over 100 random trials. The first score tells us how well each method is able to predict missing entries. The second score tells us how well the method is able to generate a tiling model for the known data.

From \Cref{tab:rw_testerror}, we see that TBMC outperforms all other algorithms as a predictive method on all five data sets. The alternating minimisation (AM) step leads to improved results in some cases, but not always. The improvement over other methods is significant in the case of the gene expression (GE) and restaurant customer (RC) data. On the other hand, in the case of the MovieLens data, the improvement gained by using TBMC over the NMF method is marginal and for both MovieLens and Netflix, the gain over simpler partitioning using simpler heuristics, \emph{average} and \emph{partition} heuristics is also marginal. 

Approximation error on the training set allows us to compare how well the database tilings of the various methods are able to capture the known data. We see from \Cref{tab:rw_testerror} that TBMC outperforms all other methods on this task for the RC and GE datasets. For MovieLens, replacing the LP rank-one approximation in TBMC with the averaging heuristic is seen to give the best performance while for ChEMBL \emph{partition} gives the best performance. These results support the use of a partitioning for the approximation task.

We speculate that the varying results for different datasets are due in large part to the degree to which each database can be modelled as a low-rank binary factorisation (tiling). It should be emphasised that the prediction algorithms considered here are restricted to those which generate interpretative binary factorisations. We are not claiming that TBMC competes favourably with all prediction algorithms (such as neural networks) if the interpretability requirement is removed. That said, our results show clearly that partitioning algorithms, and in particular TBMC, perform well compared to other algorithms for low-rank matrix completion with binary factors.

We conclude by making two further observations. Firstly, the fact that simple heuristics can perform well highlights the importance of balancing higher accuracy with constraints on computational time. Secondly, the algorithm shows evidence of over-fitting the RC data as the percentage error on the known values (training set) is $5.5$\%. 

The RC database appears to be an especially good fit for approximation by binary factors (tiles). The clustering found by TBMC on the RC database is provided as an illustration in \Cref{fig:rc}. The rows and columns have been reordered for visual effect.

\begin{table}
\centering
\resizebox{\textwidth}{!}{%
\begin{tabular}{llcccccccccll}
\cline{7-11}
 &  & \multicolumn{1}{l}{} & \multicolumn{1}{l}{} & \multicolumn{1}{l}{} & \multicolumn{1}{l|}{} & \multicolumn{5}{c|}{\cellcolor[HTML]{C0C0C0}\(\mathcal{P}\)} &  &  
 
 \\ \cline{7-11}
 &  & \multicolumn{1}{l}{} & \multicolumn{1}{l}{} & \multicolumn{1}{l}{} & \multicolumn{1}{l|}{} & \multicolumn{5}{c|}{\begin{tabular}[c]{@{}c@{}}Test (AM)\\ Train (AM)\end{tabular}} &  &  
 
 \\ \cline{3-11}
 & \multicolumn{1}{l|}{} & \multicolumn{1}{l|}{\cellcolor[HTML]{C0C0C0}\(m\)} & \multicolumn{1}{l|}{\cellcolor[HTML]{C0C0C0}\(n\)} & \multicolumn{1}{l|}{\cellcolor[HTML]{C0C0C0}\(\rho\)} & \multicolumn{1}{l|}{\cellcolor[HTML]{C0C0C0}\(\rho_D\)} & \multicolumn{1}{l|}{\cellcolor[HTML]{C0C0C0}TBMC} & \multicolumn{1}{l|}{\cellcolor[HTML]{C0C0C0}Avg} & \multicolumn{1}{l|}{\cellcolor[HTML]{C0C0C0}Partition} & \multicolumn{1}{l|}{\cellcolor[HTML]{C0C0C0}NMF} & \multicolumn{1}{l|}{\cellcolor[HTML]{C0C0C0}Spectral} &  &  
 
  \\ \cline{2-11}
\multicolumn{1}{l|}{} & \multicolumn{1}{l|}{\cellcolor[HTML]{C0C0C0}} & \multicolumn{1}{c|}{\cellcolor[HTML]{EFEFEF}} & \multicolumn{1}{c|}{\cellcolor[HTML]{EFEFEF}} & \multicolumn{1}{c|}{\cellcolor[HTML]{EFEFEF}} & \multicolumn{1}{c|}{\cellcolor[HTML]{EFEFEF}} 
& \multicolumn{1}{c|}{\begin{tabular}[c]{@{}c@{}}19.0\\ \(\bf{(15.4)}\)\end{tabular}} 
& \multicolumn{1}{c|}{\begin{tabular}[c]{@{}c@{}}30.7 \\(16.0) \end{tabular}} 
& \multicolumn{1}{c|}{\begin{tabular}[c]{@{}c@{}} 29.7 \\ (17.1) \end{tabular}} 
& \multicolumn{1}{c|}{\begin{tabular}[c]{@{}c@{}} 27.7 \end{tabular}} 
& \multicolumn{1}{c|}{\begin{tabular}[c]{@{}c@{}} 29.9 \end{tabular}} &  &

\\ \cline{7-11}
\multicolumn{1}{l|}{} & \multicolumn{1}{l|}{\multirow{-2}{*}{\cellcolor[HTML]{C0C0C0}ChEMBL}} & \multicolumn{1}{c|}{\multirow{-2}{*}{\cellcolor[HTML]{EFEFEF}7500}} & \multicolumn{1}{c|}{\multirow{-2}{*}{\cellcolor[HTML]{EFEFEF}2183}} & \multicolumn{1}{c|}{\multirow{-2}{*}{\cellcolor[HTML]{EFEFEF}0.7}} & \multicolumn{1}{c|}{\multirow{-2}{*}{\cellcolor[HTML]{EFEFEF}0.01}} 

& \multicolumn{1}{c|}{\begin{tabular}[c]{@{}c@{}}16.5\\ \bf{(14.2)}  \end{tabular}} 
& \multicolumn{1}{c|}{\begin{tabular}[c]{@{}c@{}}18.0 \\(18.0)\end{tabular}} 
& \multicolumn{1}{c|}{\begin{tabular}[c]{@{}c@{}}15.7 \\(14.7)\end{tabular}} 
& \multicolumn{1}{c|}{26.4} 
& \multicolumn{1}{c|}{30.7} &  &  

 \\ \cline{2-11}
\multicolumn{1}{l|}{} & \multicolumn{1}{l|}{\cellcolor[HTML]{C0C0C0}} & \multicolumn{1}{c|}{\cellcolor[HTML]{EFEFEF}} & \multicolumn{1}{c|}{\cellcolor[HTML]{EFEFEF}} & \multicolumn{1}{c|}{\cellcolor[HTML]{EFEFEF}} & \multicolumn{1}{c|}{\cellcolor[HTML]{EFEFEF}} & \multicolumn{1}{c|}{\begin{tabular}[c]{@{}c@{}}\bf{20.8}\\ (22.7)\end{tabular}} & \multicolumn{1}{c|}{\begin{tabular}[c]{@{}c@{}}21.3\\ (23.3)\end{tabular}} & \multicolumn{1}{c|}{\begin{tabular}[c]{@{}c@{}}33.8\\ (30.1)\end{tabular}} & \multicolumn{1}{c|}{21.1} & \multicolumn{1}{c|}{21.1} &  &  

\\ \cline{7-11}
\multicolumn{1}{l|}{} & \multicolumn{1}{l|}{\multirow{-2}{*}{\cellcolor[HTML]{C0C0C0}ML}} & \multicolumn{1}{c|}{\multirow{-2}{*}{\cellcolor[HTML]{EFEFEF}1600}} & \multicolumn{1}{c|}{\multirow{-2}{*}{\cellcolor[HTML]{EFEFEF}943}} & \multicolumn{1}{c|}{\multirow{-2}{*}{\cellcolor[HTML]{EFEFEF}0.7}} & \multicolumn{1}{c|}{\multirow{-2}{*}{\cellcolor[HTML]{EFEFEF}0.06}} & \multicolumn{1}{c|}{\begin{tabular}[c]{@{}c@{}}20.6\\ (19.5)\end{tabular}} & \multicolumn{1}{c|}{\begin{tabular}[c]{@{}c@{}}{\bf{17.7}}\\ {(19.0)}\end{tabular}} & \multicolumn{1}{c|}{\begin{tabular}[c]{@{}c@{}}33.3\\ (23.1)\end{tabular}} & \multicolumn{1}{c|}{19.5} & \multicolumn{1}{c|}{21.3} &  &

 \\ \cline{2-11}
\multicolumn{1}{l|}{} & \multicolumn{1}{l|}{\cellcolor[HTML]{C0C0C0}} & \multicolumn{1}{c|}{\cellcolor[HTML]{EFEFEF}} & \multicolumn{1}{c|}{\cellcolor[HTML]{EFEFEF}} & \multicolumn{1}{c|}{\cellcolor[HTML]{EFEFEF}} & \multicolumn{1}{c|}{\cellcolor[HTML]{EFEFEF}} & \multicolumn{1}{c|}{\begin{tabular}[c]{@{}c@{}}{\bf{19.5}}\\ (35.5)\end{tabular}} & \multicolumn{1}{c|}{\begin{tabular}[c]{@{}c@{}}22.2\\ (31.4)\end{tabular}} & \multicolumn{1}{c|}{\begin{tabular}[c]{@{}c@{}}23.3\\ (29.7)\end{tabular}} & \multicolumn{1}{c|}{22.2} & \multicolumn{1}{c|}{22.0} &  &  

\\ \cline{7-11}
\multicolumn{1}{l|}{} & \multicolumn{1}{l|}{\multirow{-2}{*}{\cellcolor[HTML]{C0C0C0}RC}} & \multicolumn{1}{c|}{\multirow{-2}{*}{\cellcolor[HTML]{EFEFEF}138}} & \multicolumn{1}{c|}{\multirow{-2}{*}{\cellcolor[HTML]{EFEFEF}130}} & \multicolumn{1}{c|}{\multirow{-2}{*}{\cellcolor[HTML]{EFEFEF}0.7}} & \multicolumn{1}{c|}{\multirow{-2}{*}{\cellcolor[HTML]{EFEFEF}0.065}} & \multicolumn{1}{c|}{\begin{tabular}[c]{@{}c@{}}\bf{4.8}\\ (5.5)\end{tabular}} & \multicolumn{1}{c|}{\begin{tabular}[c]{@{}c@{}}17.2\\ (18.2)\end{tabular}} & \multicolumn{1}{c|}{\begin{tabular}[c]{@{}c@{}}23.4\\ (17.7)\end{tabular}} & \multicolumn{1}{c|}{16.1} & \multicolumn{1}{c|}{22.2} &  &  

\\ \cline{2-11}
\multicolumn{1}{l|}{} & \multicolumn{1}{l|}{\cellcolor[HTML]{C0C0C0}} & \multicolumn{1}{c|}{\cellcolor[HTML]{EFEFEF}} & \multicolumn{1}{c|}{\cellcolor[HTML]{EFEFEF}} & \multicolumn{1}{c|}{\cellcolor[HTML]{EFEFEF}} & \multicolumn{1}{c|}{\cellcolor[HTML]{EFEFEF}} & \multicolumn{1}{c|}{\begin{tabular}[c]{@{}c@{}}11.6\\ \bf{10.9}\end{tabular}} & \multicolumn{1}{c|}{\begin{tabular}[c]{@{}c@{}}21.2\\ (21.2)\end{tabular}} & \multicolumn{1}{c|}{\begin{tabular}[c]{@{}c@{}}20.4\\ (15.2)\end{tabular}} & \multicolumn{1}{c|}{14.3} & \multicolumn{1}{c|}{21.2} &  &  

\\ \cline{7-11}
\multicolumn{1}{l|}{} & \multicolumn{1}{l|}{\multirow{-2}{*}{\cellcolor[HTML]{C0C0C0}GE}} & \multicolumn{1}{c|}{\multirow{-2}{*}{\cellcolor[HTML]{EFEFEF}5000}} & \multicolumn{1}{c|}{\multirow{-2}{*}{\cellcolor[HTML]{EFEFEF}38}} & \multicolumn{1}{c|}{\multirow{-2}{*}{\cellcolor[HTML]{EFEFEF}0.7}} & \multicolumn{1}{c|}{\multirow{-2}{*}{\cellcolor[HTML]{EFEFEF}1.0}} & \multicolumn{1}{c|}{\begin{tabular}[c]{@{}c@{}}\bf{11.3}\\ (10.4)\end{tabular}} & \multicolumn{1}{c|}{\begin{tabular}[c]{@{}c@{}}21.5\\ (21.5)\end{tabular}} & \multicolumn{1}{c|}{\begin{tabular}[c]{@{}c@{}}18.7\\ (12.2)\end{tabular}} & \multicolumn{1}{c|}{12.6} & \multicolumn{1}{c|}{21.5} &  &  

 \\ \cline{2-11}
\multicolumn{1}{l|}{} 
& \multicolumn{1}{l|}{\cellcolor[HTML]{C0C0C0}} 
& \multicolumn{1}{c|}{\cellcolor[HTML]{EFEFEF}} 
& \multicolumn{1}{c|}{\cellcolor[HTML]{EFEFEF}} 
& \multicolumn{1}{c|}{\cellcolor[HTML]{EFEFEF}} 
& \multicolumn{1}{c|}{\cellcolor[HTML]{EFEFEF}} 
& \multicolumn{1}{c|}{\begin{tabular}[c]{@{}c@{}}\bf{18.3} \\(22.8) \end{tabular}} 
& \multicolumn{1}{c|}{\begin{tabular}[c]{@{}c@{}}18.9\\(21.0)\end{tabular}} 
& \multicolumn{1}{c|}{\begin{tabular}[c]{@{}c@{}}19.9\\(18.6)\end{tabular}} 
& \multicolumn{1}{c|}{20.4} 
& \multicolumn{1}{c|}{19.9} 
&  &  

\\ \cline{7-11}
\multicolumn{1}{l|}{} 
& \multicolumn{1}{l|}{\multirow{-2}{*}{\cellcolor[HTML]{C0C0C0}Netflix}} 
& \multicolumn{1}{c|}{\multirow{-2}{*}{\cellcolor[HTML]{EFEFEF}1000}} 
& \multicolumn{1}{c|}{\multirow{-2}{*}{\cellcolor[HTML]{EFEFEF}1000}} 
& \multicolumn{1}{c|}{\multirow{-2}{*}{\cellcolor[HTML]{EFEFEF}0.7}} 
& \multicolumn{1}{c|}{\multirow{-2}{*}{\cellcolor[HTML]{EFEFEF}0.06}}
& \multicolumn{1}{c|}{\begin{tabular}[c]{@{}c@{}}\bf{15.7}\\(13.57)\end{tabular}} 
& \multicolumn{1}{c|}{\begin{tabular}[c]{@{}c@{}}17.0\\(13.9)\end{tabular}} 
& \multicolumn{1}{c|}{\begin{tabular}[c]{@{}c@{}}20.0\\(14.2)\end{tabular}} 
& \multicolumn{1}{c|}{16.4} 
& \multicolumn{1}{c|}{20.0} &  &

\\ \cline{2-11}
 &  & \multicolumn{1}{l}{} & \multicolumn{1}{l}{} & \multicolumn{1}{l}{} & \multicolumn{1}{l}{} & \multicolumn{1}{l}{} & \multicolumn{1}{l}{} & \multicolumn{1}{l}{} & \multicolumn{1}{l}{} & \multicolumn{1}{l}{} &  & 
\end{tabular}%
}
	\caption{Proportional error $\mathcal{P}$ for different datasets, on both test and training sets, with and without alternating minimisation (AM).}
	\label{tab:rw_testerror}
\end{table}

\begin{minipage}{\linewidth}

\begin{minipage}{.45\linewidth}
\centering
\includegraphics[width=0.8\linewidth]{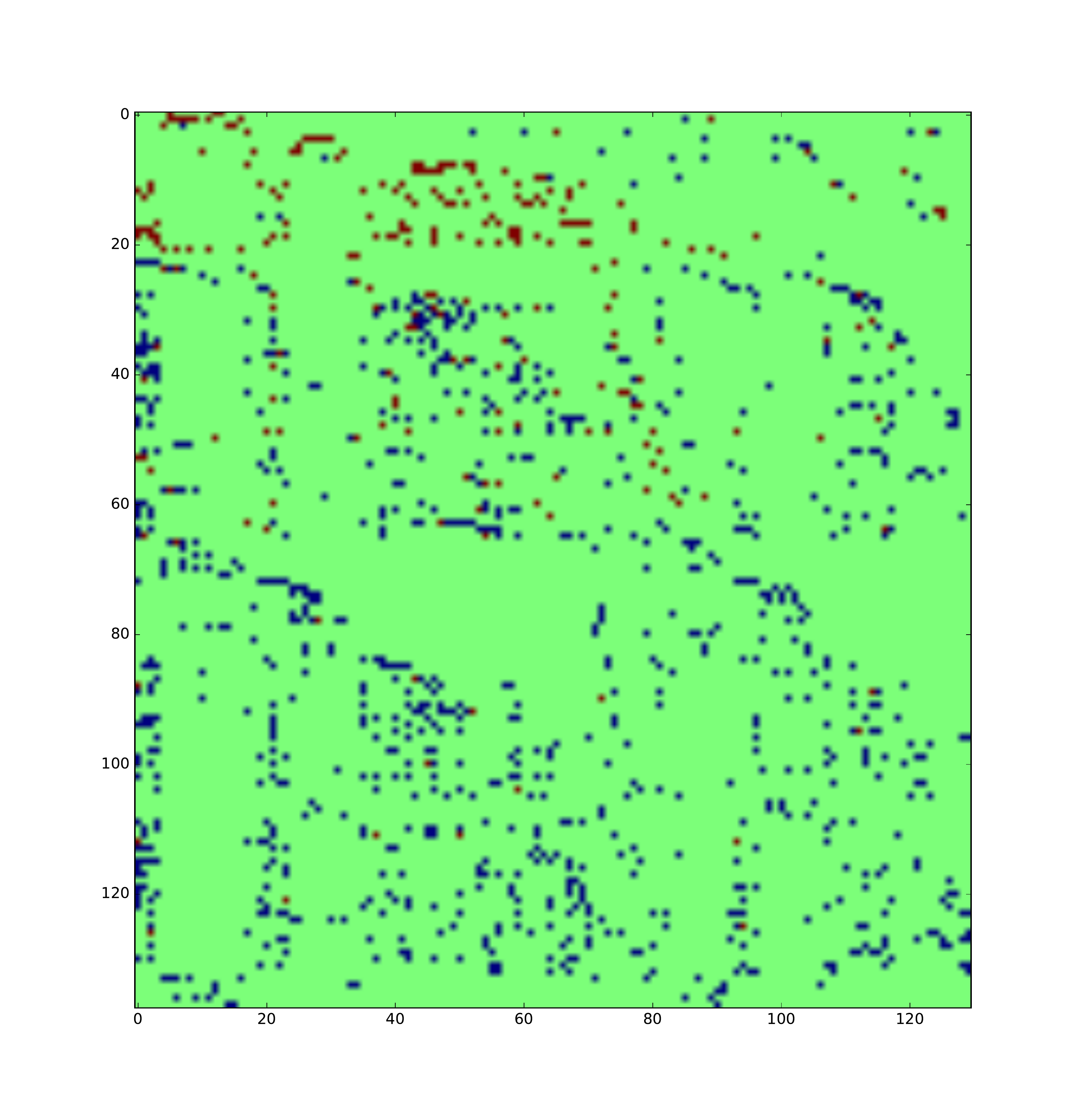}
\end{minipage}
\begin{minipage}{0.5\linewidth}
\centering
\includegraphics[width=\linewidth]{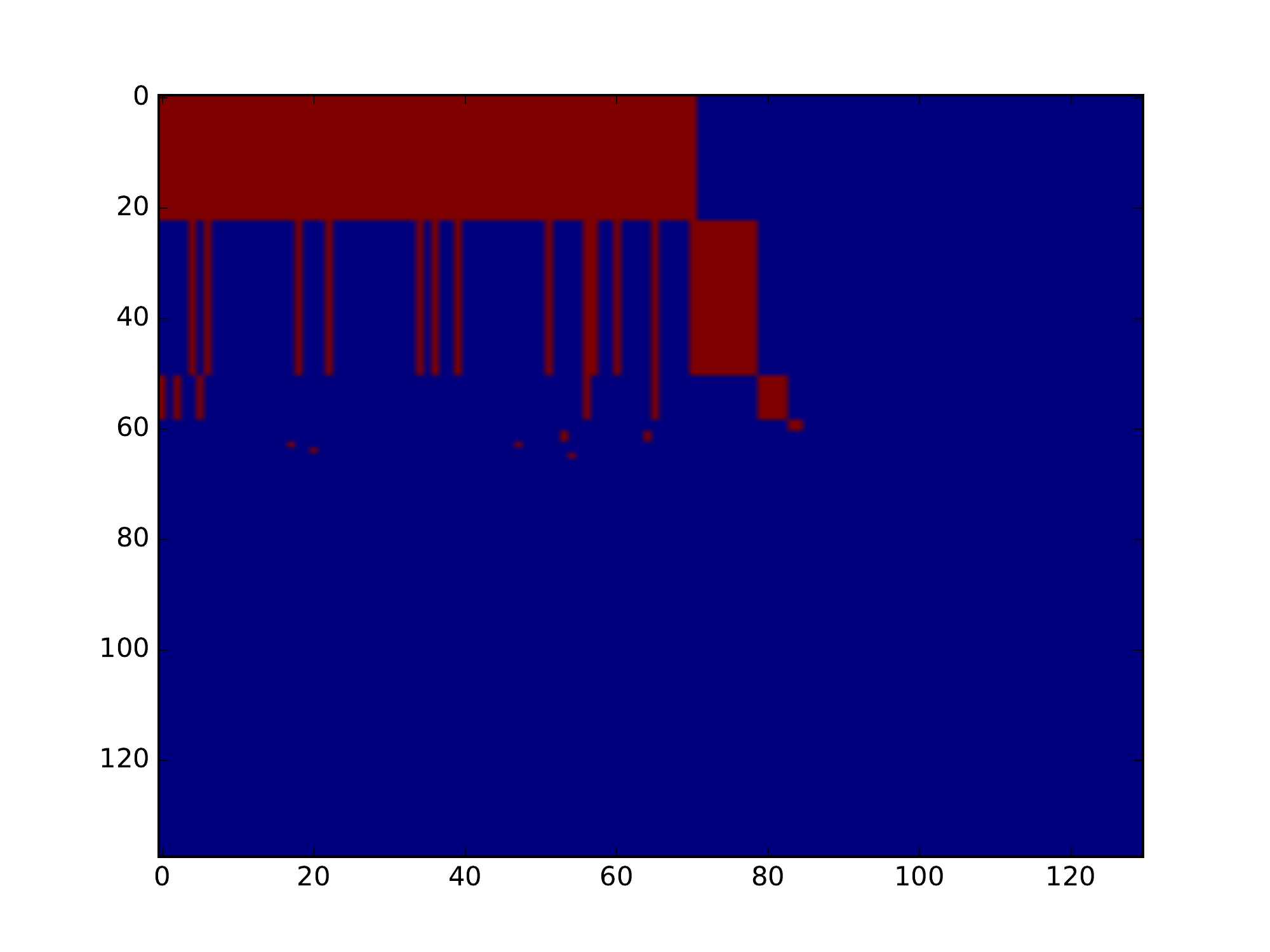}
\end{minipage}
\captionof{figure}{The RC dataset (left) with rows and columns reordered according to the factorisation found by TBMC (right).}
\label{fig:rc}
\end{minipage}

	\section{Conclusion and future directions}

Binary matrix completion for recommender systems has numerous applications. Influenced by approaches to binary matrix factorisation in the itemset mining literature, we have proposed the TBMC algorithm for binary matrix completion and shown that it outperforms alternatives on both synthetic and real data sets for certain regimes. These results make the case for the consideration of heuristic methods where typical assumptions for low rank matrix completion are violated and exact recovery is not guaranteed.

We have presented a theoretical recovery guarantee for TBMC for geometrically decaying diagonal blocks. It would interesting to extend the result to incorporate random models for missing data and noise, and also to consider alternative block models.

\bibliographystyle{plainnat}
	\bibliography{refs.bib}    

	\end{document}